\documentclass[11pt]{amsart}
\usepackage{amssymb, amsmath, latexsym,   bm, mathrsfs}
\usepackage{multirow, xcolor, float}

\usepackage[shortlabels]{enumitem}
\usepackage{tasks}
\usepackage[numbers,sort&compress]{natbib}
\usepackage{soul}
\sodef\lightspacing{}{.09em}{.2em plus .2em}{.5em plus .1em minus .1em}

\renewcommand{\baselinestretch}{\baselinestretch}
\renewcommand{\baselinestretch}{1.1}
\numberwithin{equation}{section}

\newtheorem{thm}{Theorem}[section]
\newtheorem{lem}[thm]{Lemma}
\newtheorem{cor}[thm]{Corollary}
\newtheorem{prop}[thm]{Proposition}

\newtheorem{rmk}[thm]{Remark}

\newcommand{\ord}{\operatorname{ord}}
\newcommand{\diag}{\operatorname{diag}}
\newcommand{\rank}{\operatorname{rank}}

\newcommand{\z}{\mathbb{Z}}
\newcommand{\q}{\mathbb{Q}}

\newcommand{\ind}{\operatorname{ind}}

\newcommand{\bbA}{\mathbb{A}}
\newcommand{\bbH}{\mathbb{H}}

\newcommand{\kn}{\mathfrak{n}}
\newcommand{\ko}{\mathfrak{o}}
\newcommand{\kp}{\mathfrak{p}}
\newcommand{\ks}{\mathfrak{s}}

\newcommand{\spmod}[1]{\ (\mathrm{mod}\ #1)}

\newcommand{\abs}[1]{\left\lvert #1 \right\rvert}
\newcommand{\ang}[1]{\left\langle #1 \right\rangle}

\begin{document}

\title[Minimal rank of P$n$U quadratic forms over local fields]{Minimal rank of primitively $n$-universal integral quadratic forms over local rings}

\author{Byeong-Kweon Oh and Jongheun Yoon}

\address{Department of Mathematical Sciences and Research Institute of Mathematics, Seoul National University, Seoul 08826, Korea}
\email{bkoh@snu.ac.kr}
\thanks{This work was supported by the National Research Foundation of Korea(NRF) grant funded by the Korea government(MSIT)(RS-2024-00342122) and (NRF-2020R1A5A1016126). }

\address{Department of Mathematical Sciences, Seoul National University, Seoul 08826, Korea}
\email{jongheun.yoon@snu.ac.kr}



\begin{abstract} Let $F$ be a local field and let $R$ be its ring of integers.  
For a positive integer $n$, an integral quadratic form defined over $R$ is called primitively $n$-universal if it primitively represents all quadratic forms of rank $n$.  It was proved in \cite{EG21rj} that the minimal rank of primitively $1$-universal quadratic forms over the $p$-adic integer ring $\z_p$ is $2$ if $p$ is odd, and $3$ otherwise.
  In this article, we completely determine the minimal rank of primitively $n$-universal quadratic forms over $R$ for any positive integer $n$ and any local ring $R$ such that $2$ is a unit or a prime.
\end{abstract}

\maketitle

%%%%%%%%%%%%%%%%%%%%%%%%%%%%%%%%%%%%%%%%%%%%%%%%%%%%%%%%%%%%%%%%%%%%%%%%%%%%%%%%
%%%%%%%%%%%%%%%%%%%%%%%%%%%%%%%%%%%%%%%%%%%%%%%%%%%%%%%%%%%%%%%%%%%%%%%%%%%%%%%%
%%%%%%%%%%%%%%%%%%%%%%%%%%%%%%%%%%%%%%%%%%%%%%%%%%%%%%%%%%%%%%%%%%%%%%%%%%%%%%%%
%%%%%%%%%%%%%%%%%%%%%%%%%%%%%%%%%%%%%%%%%%%%%%%%%%%%%%%%%%%%%%%%%%%%%%%%%%%%%%%%
%%%%%%%%%%%%%%%%%%%%%%%%%%%%%%%%%%%%%%%%%%%%%%%%%%%%%%%%%%%%%%%%%%%%%%%%%%%%%%%%
%%%%%%%%%%%%%%%%%%%%%%%%%%%%%%%%%%%%%%%%%%%%%%%%%%%%%%%%%%%%%%%%%%%%%%%%%%%%%%%%
%%%%%%%%%%%%%%%%%%%%%%%%%%%%%%%%%%%%%%%%%%%%%%%%%%%%%%%%%%%%%%%%%%%%%%%%%%%%%%%%
%-------------------------------------------------------------------------------
\section{Introduction}
%-------------------------------------------------------------------------------
\label{sec:intro}
%%%%%%%%%%%%%%%%%%%%%%%%%%%%%%%%%%%%%%%%%%%%%%%%%%%%%%%%%%%%%%%%%%%%%%%%%%%%%%%%
%%%%%%%%%%%%%%%%%%%%%%%%%%%%%%%%%%%%%%%%%%%%%%%%%%%%%%%%%%%%%%%%%%%%%%%%%%%%%%%%
%%%%%%%%%%%%%%%%%%%%%%%%%%%%%%%%%%%%%%%%%%%%%%%%%%%%%%%%%%%%%%%%%%%%%%%%%%%%%%%%
%%%%%%%%%%%%%%%%%%%%%%%%%%%%%%%%%%%%%%%%%%%%%%%%%%%%%%%%%%%%%%%%%%%%%%%%%%%%%%%%
%%%%%%%%%%%%%%%%%%%%%%%%%%%%%%%%%%%%%%%%%%%%%%%%%%%%%%%%%%%%%%%%%%%%%%%%%%%%%%%%
%%%%%%%%%%%%%%%%%%%%%%%%%%%%%%%%%%%%%%%%%%%%%%%%%%%%%%%%%%%%%%%%%%%%%%%%%%%%%%%%
%%%%%%%%%%%%%%%%%%%%%%%%%%%%%%%%%%%%%%%%%%%%%%%%%%%%%%%%%%%%%%%%%%%%%%%%%%%%%%%%

%\newdimen\ofontdimentwo%
%\newdimen\ofontdimensam%
%\newdimen\ofontdimennes%
%\newdimen\ofontdimenqil%
%\ofontdimentwo=\fontdimen2\font%
%\ofontdimensam=\fontdimen3\font%
%\ofontdimennes=\fontdimen4\font%
%\ofontdimenqil=\fontdimen7\font%
%\fontdimen2\font=3.3pt%
%\fontdimen3\font=1.65pt%
%\fontdimen4\font=1.1pt%
%\fontdimen7\font=1.1pt%
%\fontdimen2\font=\ofontdimentwo%
%\fontdimen3\font=\ofontdimensam%
%\fontdimen4\font=\ofontdimennes%
%\fontdimen7\font=\ofontdimenqil%

In $1770$, Lagrange proved that the quadratic form $x^2 + y^2 + z^2 + w^2$ represents all positive integers, that is, the diophantine equation 
\[
x^2 + y^2 + z^2 + w^2=n
\]
 has an integer solution $x,y,z$, and $w$ for any positive integer $n$. Any (positive definite and integral) quadratic form having this property is called \emph{universal}, as introduced by Dickson and Ross. The concept of universality has been generalized in many directions, including the $n$-universality and the primitive $n$-universality over various rings. Among those, the primitive $n$-universality over local rings is of particular interest in this article, whose precise definition will be given in what follows.

Let $F$ be a local field at place $\kp$ and let $R$ be the ring of integers in $F$. We generally assume that the characteristic of $F$ is not two. A quadratic form of rank $n$ defined over $F$ is a quadratic homogeneous polynomial
\[
 f(x_1, \dotsc, x_n) = \sum_{i, j = 1}^n f_{ij}x_i x_j \qquad \text{($f_{ij} = f_{ji}\in F$),}
\]
where the corresponding symmetric matrix $M_f = (f_{ij})$, which is called the Gram matrix of $f$, is nondegenerate. We say that $f$ is classically integral if the Gram matrix $M_f$ of $f$ is an integral matrix, that is, $f_{ij} = f_{ji} \in R$ for any $i$, $j$. Whereas, $f$ is called non-classically integral if $f$ is an integral polynomial, that is, $f_{ii}$ and $2f_{ij}$ are in $R$ for any $i$, $j$.

For two (integral) quadratic forms $f$ and $g$ of rank $n$ and $m$, respectively, we say that $f$ is represented by $g$ over $R$ if there is an integral matrix $T\in M_{n,m}(R)$ such that
\[
 M_f = T^t M_g T\text.
\]
We say that $f$ is isometric to $g$ if the above integral matrix $T$ is invertible. We say that $f$ is primitively represented by $g$ if the above integral matrix $T$ can be extended to an invertible matrix in $GL_m(R)$ by adding suitable $(m-n)$ columns. In particular, an integer $k \in R$ is primitively represented by $g$ if and only if there are integers $x_1$, \dots, $x_m\in R$ such that
\[
k = g(x_1, \dotsc, x_m) \text{\, and at least one of $x_i$ is a unit.}
\]

For a positive integer $n$, an integral quadratic form is called (\emph{primitively}) $n$-\emph{universal} if it (primitively, respectively) represents all integral quadratic forms of rank $n$. A complete classification of $n$-universal quadratic forms is recently published in \cite{He23}, \cite{HH23}, and \cite{HHX23}. Finding primitively $1$-universal quadratic forms was first considered, as far as the authors know, by Budarina in \cite{Bu10}. She classified all primitively $1$-universal quaternary quadratic forms when $R = \z_p$ for any odd prime $p$, and also obtained some partial classification when $R = \z_2$. Recently, Earnest and Gunawardana provided a complete classification of all primitively $1$-universal quadratic forms over $\z_p$ for any prime $p$, including results for both classically integral and non-classically integral forms when $p=2$. Budarina also classified in \cite{Bu11} primitively $2$-universal quadratic forms having squarefree odd discriminant.

The subsequent discussion will be conducted in a more suitable language of quadratic spaces and lattices. A quadratic $F$-space, simply a space, is a finite dimensional vector space $V$ over $F$ equipped with a symmetric bilinear form
\[
 B : V\times V\to F\text,
\]
and the corresponding quadratic form $Q: V\to F$ is defined by $Q(v) = B(v, v)$ for any $v\in V$. A quadratic $R$-lattice, simply a lattice, is a finitely generated $R$-submodule of a quadratic $F$-space. The scale $\ks L$ of a quadratic $R$-lattice $L$ is the $R$-submodule $B(L, L)$ of $F$, and the norm $\kn L$ of $L$ is the $R$-submodule of $F$ generated by $Q(L) = \{ Q(v) \mid v\in L\}$. An $R$-lattice $L$ is called integral if $\ks L\subseteq R$. Throughout this article, we always assume that any lattice is integral, unless stated otherwise. An (integral) $R$-lattice $L$ is called even if $\kn L\subseteq 2R$.

Given a basis $e_1, \dotsc, e_n$ for a space $V$, the corresponding symmetric $n\times n$ matrix $M_V = (B(e_i, e_j))$ is called the Gram matrix of $V$, and in this case, we write
\[
 V\cong (B(e_i, e_j))\quad\text{in}\quad e_1, \dotsc, e_n\text.
\]
The discriminant $dV$ of a space $V$ is an element of the quotient monoid $F / (F^\times)^2$ defined by $dV = (\det M_V)(F^\times)^2$. We call a space $V$ nondegenerate if $M_V$ is nondegenerate. If a space $V$ has the zero Gram matrix, then $V$ is called totally isotropic. For a nondegenerate space $V$, all maximal totally isotropic subspace of $V$ has the same dimension. Such a dimension is denoted by $\ind V$, called the Witt index of $V$. One may easily show that if $\ind V\ge r$, then $V$ is orthogonally split by a $2r$-dimensional hyperbolic space.

Given a basis $e_1, \dotsc, e_n$ for a lattice $L$, the corresponding symmetric $n\times n$ matrix $M_L = (B(e_i, e_j))$ is called the Gram matrix of $L$, and in this case, we write
\[
 L\cong (B(e_i, e_j))\quad\text{in}\quad e_1, \dotsc, e_n\text.
\]
The discriminant $dL$ of a lattice $L$ is an element of the quotient monoid $F / (R^\times)^2$ defined by $dL = (\det M_L)(R^\times)^2$. We call a lattice $L$ nondegenerate if $M_L$ is nondegenerate. Throughout this article, we always assume that a lattice is nondegenerate, unless stated otherwise. For a symmetric $n\times n$ matrix $N$ over $R$, we let $\langle N\rangle$ (or simply $N$) stand for any $R$-lattice whose Gram matrix is $N$. Moreover, for $\alpha_1$, \dots, $\alpha_n\in R$,  the notation
\[
 \langle \alpha_1, \dotsc, \alpha_n \rangle
\]
stands for the $R$-lattice having a diagonal Gram matrix $\diag(\alpha_1, \dotsc, \alpha_n)$. Let $\Delta = 1 - 4\rho$ be a fixed nonsquare unit in $R$ such that $\rho\in R^\times$. The symbols $\mathbb{H}$ and $\mathbb{A}$ stand for binary $R$-lattices whose Gram matrices are
\[
 \mathbb{H}\cong \left(\begin{smallmatrix}0&1\\1&0\end{smallmatrix}\right) \quad\text{and}\quad \mathbb{A}\cong \left(\begin{smallmatrix}2&1\\1&2\rho\end{smallmatrix}\right), \ \text{respectively.}
\]
In fact, $\mathbb{H}$ is the isotropic even unimodular lattice, and $\mathbb{A}$ is the anisotropic even unimodular lattice. Scalings of $\bbH$ and $\bbA$ by an integer $\alpha\in R$ are denoted by $\alpha\bbH\cong \left(\begin{smallmatrix} 0 & \alpha \\ \alpha & 0\end{smallmatrix}\right)$ and $\alpha\bbA\cong \left(\begin{smallmatrix} 2\alpha & \alpha \\ \alpha & 2\rho\alpha \end{smallmatrix}\right)$. For a positive integer $n$, $\bbH^n$ denotes the lattice $\bbH \mathbin{\perp} \dotsb \mathbin{\perp} \bbH$ with rank $2n$. We also use the same symbol $\mathbb{H}$ which stands for the hyperbolic plane $F\mathbb{H}$, if no confusion arises. 

For two $R$-lattices $L$ and $M$, a representation from $L$ to $M$ is a linear map $\sigma : L \to M$ such that
\[
 B(\sigma(v), \sigma(v')) = B(v, v')
\]
for any $v$, $v'\in L$.  In addition, if $\sigma(L)$ is a direct summand of $M$, then such a representation is called \emph{primitive}. If there exists a (primitive) representation $\sigma : L \to M$, then we say that $L$ is (primitively, respectively) represented by $M$. If such a representation $\sigma$ is bijective then we say that $L$ is isometric to $M$, denoted by $\sigma: L\cong M$. An $R$-lattice $M$ is called (\emph{primitively}) $n$-\emph{universal} if $M$ (primitively, respectively) represents all $R$-lattices of rank $n$.  We denote by $u_R(n)$ the minimal rank of $n$-universal $R$-lattices, and denote by $u^\ast_R(n)$ the minimal rank of primitively $n$-universal (P$n$U, for short) $R$-lattices. For $u_R(n)$, the followings are well-known (see, for instance, \cite{He23, HH23, HHX23}): If $R$ is nondyadic, then
\[
 u_R(n) = \begin{cases}
  2n & \text{if $1\le n\le 2$,}\\
  n+3 & \text{if $n\ge 3$;}
 \end{cases}
\]
If $R$ is unramified dyadic, then
\[
 u_R(n) = \begin{cases}
  2n+1 & \text{if $1\le n\le 2$,}\\
  n+3 & \text{if $n\ge 3$.}
 \end{cases}
\]

The aim of this article is to determine $u^\ast_R(n)$ completely for any integer $n$ and any local ring $R$ such that $2$ is a unit or a prime.

Let $I$ be an ideal in $R$. For any $\alpha$, $\beta\in R$, the congruence $\alpha \equiv \beta \spmod{I}$ takes on its usual meaning, namely $\alpha - \beta\in I$. For any $\gamma\in R$, we write
\[
 \alpha \equiv \beta \pmod{\gamma}
\]
instead of $\alpha \equiv \beta \spmod{\gamma R}$. Moreover, for any subset $S$, $T$ of $R$, the congruence $S \equiv T \spmod{I}$ means that $\alpha \equiv \beta \spmod{I}$ for any $\alpha\in S$, $\beta\in T$. For instance, we have $(\z_2^\times)^2 \equiv 1\spmod8$.
%In Section 2, we establish a necessary condition for primitive $n$-universality of an $R$-lattice. Building on this result, Sections 3 through 6 determine the minimal rank of $R$-lattices possessing the property, with each section dedicated to different cases based on the nature of the ring $R$. Section 3 deals with nondyadic $R$ case, while the next three sections are about unramified dyadic $R$ cases. Section 4 addresses the problem for non-classically integral lattices, whereas in the remaining two sections we return to classically integral lattices. Section 5 resolves $R\ne \z_2$ case. Section 6 is solely for $\z_2$-lattices.

Any unexplained notation and terminology can be found in \cite{Ki} or \cite{OM}.

%%%%%%%%%%%%%%%%%%%%%%%%%%%%%%%%%%%%%%%%%%%%%%%%%%%%%%%%%%%%%%%%%%%%%%%%%%%%%%%%
%%%%%%%%%%%%%%%%%%%%%%%%%%%%%%%%%%%%%%%%%%%%%%%%%%%%%%%%%%%%%%%%%%%%%%%%%%%%%%%%
%%%%%%%%%%%%%%%%%%%%%%%%%%%%%%%%%%%%%%%%%%%%%%%%%%%%%%%%%%%%%%%%%%%%%%%%%%%%%%%%
%%%%%%%%%%%%%%%%%%%%%%%%%%%%%%%%%%%%%%%%%%%%%%%%%%%%%%%%%%%%%%%%%%%%%%%%%%%%%%%%
%%%%%%%%%%%%%%%%%%%%%%%%%%%%%%%%%%%%%%%%%%%%%%%%%%%%%%%%%%%%%%%%%%%%%%%%%%%%%%%%
%%%%%%%%%%%%%%%%%%%%%%%%%%%%%%%%%%%%%%%%%%%%%%%%%%%%%%%%%%%%%%%%%%%%%%%%%%%%%%%%
%%%%%%%%%%%%%%%%%%%%%%%%%%%%%%%%%%%%%%%%%%%%%%%%%%%%%%%%%%%%%%%%%%%%%%%%%%%%%%%%
%-------------------------------------------------------------------------------
\section{Witt indices of spaces admitting a P$n$U lattice}
%-------------------------------------------------------------------------------
\label{sec:general}
%%%%%%%%%%%%%%%%%%%%%%%%%%%%%%%%%%%%%%%%%%%%%%%%%%%%%%%%%%%%%%%%%%%%%%%%%%%%%%%%
%%%%%%%%%%%%%%%%%%%%%%%%%%%%%%%%%%%%%%%%%%%%%%%%%%%%%%%%%%%%%%%%%%%%%%%%%%%%%%%%
%%%%%%%%%%%%%%%%%%%%%%%%%%%%%%%%%%%%%%%%%%%%%%%%%%%%%%%%%%%%%%%%%%%%%%%%%%%%%%%%
%%%%%%%%%%%%%%%%%%%%%%%%%%%%%%%%%%%%%%%%%%%%%%%%%%%%%%%%%%%%%%%%%%%%%%%%%%%%%%%%
%%%%%%%%%%%%%%%%%%%%%%%%%%%%%%%%%%%%%%%%%%%%%%%%%%%%%%%%%%%%%%%%%%%%%%%%%%%%%%%%
%%%%%%%%%%%%%%%%%%%%%%%%%%%%%%%%%%%%%%%%%%%%%%%%%%%%%%%%%%%%%%%%%%%%%%%%%%%%%%%%
%%%%%%%%%%%%%%%%%%%%%%%%%%%%%%%%%%%%%%%%%%%%%%%%%%%%%%%%%%%%%%%%%%%%%%%%%%%%%%%%

In this section, we prove some necessary condition for a quadratic space having a primitively $n$-universal  $R$-lattice. More precisely, we prove that if $M$ is a primitively $n$-universal $R$-lattice, then the space $FM$ represents a $2n$-dimensional hyperbolic space. In particular, we have $u^\ast_R(n)\ge 2n$ for any positive integer $n$.
In fact, the result follows from an easy application of a multivariable version of Hensel's lemma given below. Let $K$ be a complete field with respect to an absolute value $\abs{\cdot}$ satisfying the strong triangle inequality, and let $\ko := \{x\in K: \abs{x}\le 1\}$.

Let $n\ge 1$ and define a norm of $\mathbf{c} = (c_1, \dots, c_n)\in K^n$ by $\lVert \mathbf{c} \rVert := \max_i \abs{c_i}$. Denote the derivative matrix and Jacobian of $\mathbf{f} = (f_1(\mathbf{X}), \dots, f_n(\mathbf{X})) \in K[X_1, \dots, X_n]^n$ by 
\[
 (D\mathbf{f})(\mathbf{X}) = \left(\frac{\partial f_i}{\partial X_j}\right)_{1\le i, j\le n} \quad \text{and} \quad J_\mathbf{f}(\mathbf{X}) = \det((D\mathbf{f})(\mathbf{X})).
\]

\begin{thm} \label{con1}
Let $\mathbf{f} \in \ko[\mathbf{X}]^n$ and let $\mathbf{a} \in \ko^n$ satisfy $\lVert \mathbf{f}(\mathbf{a}) \rVert < \abs{J_\mathbf{f}(\mathbf{a})}^2$. Then there is a unique $\bm{\alpha}\in\ko^n$ such that $\mathbf{f}(\bm{\alpha}) = \mathbf{0}$ and $\lVert \bm{\alpha} - \mathbf{a} \rVert < \abs{J_\mathbf{f}(\mathbf{a})}$.
\end{thm}

\begin{proof}
See Theorem 3.3 of \cite{Co}. 
\end{proof}

\begin{thm} \label{con2}
For $m\ge n\ge 1$, let $\mathbf{f}(\mathbf{X}) = (f_1, \dots, f_n)\in \ko[X_1, \dots, X_m]^n$ and $\mathbf{a} = (a_1, \dots, a_m)\in \ko^m$ satisfy $\lVert \mathbf{f}(\mathbf{a}) \rVert < \abs{J_{\mathbf{f}, n}(\mathbf{a})}^2$, where $J_{\mathbf{f}, n}(\mathbf{a}) = \det\left(\frac{\partial f_i}{\partial X_j}\right)_{1\le i, j\le n}$. Then there is an $\bm{\alpha}\in\ko^n$ such that
\[
 \mathbf{f}(\alpha_1, \dots, \alpha_n, a_{n+1}, \dots, a_m) = \mathbf{0}
\]
and $\abs{\alpha_i - a_i} < \abs{J_{\mathbf{f},n}(\mathbf{a})}$ for $i = 1, \dots, n$.
\end{thm}

\begin{proof}
See Theorem 3.8 of \cite{Co}. 
\end{proof}

\begin{lem}
For $m\ge n\ge 1$, let $G = (g_{ij})$ and $H$ be symmetric matrices over $\ko$ of size $n$ and $m$, respectively, and let $A = (\mathbf{a}_1, \dots, \mathbf{a}_n) = (a_{ij})_{m\times n}$ be a matrix over $\ko$ such that $G = A^t H A$. Suppose that $H$ is nondegenerate and $A$ is primitive. Then, for any $(\gamma_1, \dots, \gamma_n)\in\ko^n$ satisfying
\[
 \max_{1\le i\le n} \abs{g_{in} - \gamma_i} < 4\abs{\det H}^2\text{,}
\]
there is an $\bm{\alpha}\in\ko^m$ such that $A_1 = (\mathbf{a}_1, \dots, \mathbf{a}_{n-1}, \bm{\alpha})$ is also primitive, and $(g'_{ij}) = A_1^t F A_1$ satisfies
\[
 g'_{ij} = \begin{cases}
  \gamma_j & \text{if }i = n\text{,}\\
  \gamma_i & \text{if }j = n\text{,}\\
  g_{ij} & \text{otherwise.}
 \end{cases}
\]
\end{lem}

\begin{proof}
We regard\, $g_{1n} - \gamma_1$, \,\dots, \,$g_{n-1, n} - \gamma_{n-1}$,\, and\, $g_{nn} - \gamma_n$ as $n$ polynomials in $m$ variables $a_{1n}, \dots, a_{mn}$. Note that $g_{in} - \gamma_i = g_{ni} - \gamma_i$ for $1\le i\le n$. To apply Theorem \ref{con2}, it suffices to show that the derivative matrix has an $n\times n$ subdeterminant whose absolute value is greater than or equal to $2\abs{\det H}$.

First, note that the derivative matrix is $\diag(1, \dots, 1, 2) A^t H$. Hence, it suffices to show that $A^t H$ has an $n\times n$ subdeterminant whose absolute value is greater than or equal to $\abs{\det H}$. By the primitivity of $A$, we may complete $A^t$ to an element of $GL_m(\ko)$, namely
\[
 U = \begin{pmatrix}A^t \\\hline C\end{pmatrix}\text{.}
\]
By the theory of modules over a principal ideal domain, there is a $V\in GL_m(\ko)$ such that $U H V = T = (t_{ij})_{m\times m}$ is lower triangular. Clearly, $A^t H V = (t_{ij})_{n\times m}$ has a submatrix (the leftmost one) whose determinant is $d = \prod_{i=1}^n t_{ii}$. Then, 
\[
 \abs{d} = \abs{\prod_{i=1}^n t_{ii}} \ge \abs{\prod_{i=1}^m t_{ii}} = \abs{\det UHV} = \abs{\det H}\text.
\]
Now, observe that $d$ is a linear combination of $n\times n$ subdeterminants of $A^t H$. Hence, there exists at least one with absolute value greater than or equal to $\abs{\det H}$. This proves the lemma.
\end{proof}

From now until the end of this section, we do not assume that a lattice is nondegenerate.

\begin{cor}
Let $M$ be a nondegenerate $R$-lattice. Let $N$ and $N'$ be $R$-lattices of the same rank $n$ with Gram matrices $G$, $G'$, respectively, such that $G - G' \in M_n (4(dM)^2 \kp)$. Then $N$ is primitively represented by $M$ if and only if $N'$ is primitively represented by $M$.
\end{cor}

\begin{cor}
Let $M$ be a nondegenerate $R$-lattice. Then the followings are equivalent:
\begin{enumerate}[\textup{(\arabic{enumi})}]
\item $\ind FM \ge n$;
\item some lattice $N$ of rank $n$ with $\ks N \subseteq 4(dM)^2 \kp$ is primitively represented by $M$;
\item all lattices $N$ of rank $n$ with $\ks N \subseteq 4(dM)^2 \kp$ are primitively represented by $M$.
\end{enumerate}
\end{cor}

\begin{thm}\label{thm:indgen}
If $M$ is a primitively $n$-universal nondegenerate $R$-lattice, then $\ind FM \ge n$.  In particular, $u^\ast_R(n)\ge 2n$.
\end{thm}

\begin{proof}
This is a direct consequence of the corollary given above.
\end{proof}

\begin{cor}\label{cor:corindgen}
 Let $M$ be a primitively $n$-universal $R$-lattice of rank $m\ge 2n$. We have the followings:

\noindent \textup{(a)} If $m = 2n$, then $FM$ is hyperbolic.

\noindent \textup{(b)} If $m = 2n+1$, then $FM \cong \bbH^n\mathbin{\perp}\ang{(-1)^n dM}$.

\noindent \textup{(c)} If $m = 2n+2$ and $dM = (-1)^{n+1}$, then $FM$ is hyperbolic.
\end{cor}

As a direct consequence of the above corollary, any primitively $n$-universal $R$-lattice of rank $2n$ is an $n$-universal $R$-lattice on the hyperbolic space $\bbH^n$.  Furthermore, any primitively $n$-universal $R$-lattice $M$ of rank $2n+1$ is an $n$-universal $R$-lattice on the space $\bbH^n\mathbin{\perp}\ang{(-1)^n dM}$.

\begin{rmk}
Suppose that an $R$-lattice $M$ has a Jordan splitting $M = M_1 \mathbin{\perp} \dotsb \mathbin{\perp} M_r$, where $R \supseteq \ks M_1 \supsetneq \cdots \supsetneq \ks M_r$. It can be shown that if $M$ primitively represents an $R$-lattice $N$ of rank $n$ with $\ks N \subsetneq \ks M_r$, then $\rank M\ge 2n$ \textup(for the proof, see \textup{\cite[Lemma~5.5]{CO23})}. 
%\textup{Theorem~\ref{thm:indgen}} actually provides a necessary space condition of being primitively $n$-universal, which is stronger than a rank condition.
\end{rmk}

%%%%%%%%%%%%%%%%%%%%%%%%%%%%%%%%%%%%%%%%%%%%%%%%%%%%%%%%%%%%%%%%%%%%%%%%%%%%%%%%
%%%%%%%%%%%%%%%%%%%%%%%%%%%%%%%%%%%%%%%%%%%%%%%%%%%%%%%%%%%%%%%%%%%%%%%%%%%%%%%%
%%%%%%%%%%%%%%%%%%%%%%%%%%%%%%%%%%%%%%%%%%%%%%%%%%%%%%%%%%%%%%%%%%%%%%%%%%%%%%%%
%%%%%%%%%%%%%%%%%%%%%%%%%%%%%%%%%%%%%%%%%%%%%%%%%%%%%%%%%%%%%%%%%%%%%%%%%%%%%%%%
%%%%%%%%%%%%%%%%%%%%%%%%%%%%%%%%%%%%%%%%%%%%%%%%%%%%%%%%%%%%%%%%%%%%%%%%%%%%%%%%
%%%%%%%%%%%%%%%%%%%%%%%%%%%%%%%%%%%%%%%%%%%%%%%%%%%%%%%%%%%%%%%%%%%%%%%%%%%%%%%%
%%%%%%%%%%%%%%%%%%%%%%%%%%%%%%%%%%%%%%%%%%%%%%%%%%%%%%%%%%%%%%%%%%%%%%%%%%%%%%%%
%-------------------------------------------------------------------------------
\section{Minimal rank of P$n$U lattices over nondyadic local rings}
%-------------------------------------------------------------------------------
\label{sec:podd}
%%%%%%%%%%%%%%%%%%%%%%%%%%%%%%%%%%%%%%%%%%%%%%%%%%%%%%%%%%%%%%%%%%%%%%%%%%%%%%%%
%%%%%%%%%%%%%%%%%%%%%%%%%%%%%%%%%%%%%%%%%%%%%%%%%%%%%%%%%%%%%%%%%%%%%%%%%%%%%%%%
%%%%%%%%%%%%%%%%%%%%%%%%%%%%%%%%%%%%%%%%%%%%%%%%%%%%%%%%%%%%%%%%%%%%%%%%%%%%%%%%
%%%%%%%%%%%%%%%%%%%%%%%%%%%%%%%%%%%%%%%%%%%%%%%%%%%%%%%%%%%%%%%%%%%%%%%%%%%%%%%%
%%%%%%%%%%%%%%%%%%%%%%%%%%%%%%%%%%%%%%%%%%%%%%%%%%%%%%%%%%%%%%%%%%%%%%%%%%%%%%%%
%%%%%%%%%%%%%%%%%%%%%%%%%%%%%%%%%%%%%%%%%%%%%%%%%%%%%%%%%%%%%%%%%%%%%%%%%%%%%%%%
%%%%%%%%%%%%%%%%%%%%%%%%%%%%%%%%%%%%%%%%%%%%%%%%%%%%%%%%%%%%%%%%%%%%%%%%%%%%%%%%

In this section, we prove that $u^\ast_R(n) = 2n$ for any nondyadic local ring $R$ and any positive integer $n$. Furthermore, it is shown that, up to isometry, there exist exactly one primitively $2$-universal lattice of rank $4$ and exactly two primitively $3$-universal lattices of rank $6$ over any nondyadic local ring.

Recall that $\bbH\cong \left(\begin{smallmatrix}0&1\\1&0\end{smallmatrix}\right)$ is the even unimodular $R$-lattice on the hyperbolic plane. Let $x, y$ be a basis for $\bbH$ such that $\bbH\cong \left(\begin{smallmatrix}0&1\\1&0\end{smallmatrix}\right)$ in this basis. Then, $Q(\alpha x + \beta y) = 2\alpha\beta$ for any $\alpha$, $\beta\in R$. Now, the proof of the following lemma is quite straightforward.

\begin{lem}\label{lem:easypnu}
\textup{(a)} The $R$-lattice $\bbH$ primitively represents all even integers. In particular, if $R$ is nondyadic, then $\bbH$ is primitively $1$-universal.

\noindent \textup{(b)} If an $R$-lattice $J$ primitively represents a lattice $K$ of rank $k$, then $\bbH\mathbin{\perp} J$ primitively represents all lattices of rank $k+1$ of the form $\ang{\alpha}\mathbin{\perp} K$ for any $\alpha\in 2R$. In particular, if $R$ is nondyadic, then $\bbH^n$ is primitively $n$-universal.
\end{lem}

The above lemma shows that, for any nondyadic local ring $R$, $\bbH^n$ is a primitively $n$-universal $R$-lattice of rank $2n$. By combining it with Corollary~\ref{cor:corindgen}, we conclude that the minimal rank of primitively $n$-universal quadratic $R$-lattices is exactly $2n$.

\begin{prop}
Let $R$ be a nondyadic local ring and let $\pi\in R$ be a prime.

\noindent \textup{(a)} Any primitively $2$-universal $R$-lattice of rank $4$ is isometric to $\bbH^2$.

\noindent \textup{(b)} Any primitively $3$-universal $R$-lattice of rank $6$ is isometric to either $\bbH^3$ or $\bbH^2\mathbin{\perp} \pi\bbH$.
\end{prop}

\begin{proof}
(a) According to \cite[Proposition~3.3]{HHX23} and Corollary~\ref{cor:corindgen}, any $2$-universal quaternary lattice on the hyperbolic space $\bbH^2$ is isometric to $I_4\cong \bbH^2$. (b) According to \cite[Proposition~3.4]{HHX23} and Corollary~\ref{cor:corindgen}, any $3$-universal senary lattice on the hyperbolic space $\bbH^3$ is isometric to either $\bbH^3$ or $\bbH^2\mathbin{\perp} \pi\bbH$. It can be easily shown that $\bbH^2\mathbin{\perp} \pi\bbH$ is also primitively $3$-universal.
\end{proof}

%%%%%%%%%%%%%%%%%%%%%%%%%%%%%%%%%%%%%%%%%%%%%%%%%%%%%%%%%%%%%%%%%%%%%%%%%%%%%%%%
%%%%%%%%%%%%%%%%%%%%%%%%%%%%%%%%%%%%%%%%%%%%%%%%%%%%%%%%%%%%%%%%%%%%%%%%%%%%%%%%
%%%%%%%%%%%%%%%%%%%%%%%%%%%%%%%%%%%%%%%%%%%%%%%%%%%%%%%%%%%%%%%%%%%%%%%%%%%%%%%%
%%%%%%%%%%%%%%%%%%%%%%%%%%%%%%%%%%%%%%%%%%%%%%%%%%%%%%%%%%%%%%%%%%%%%%%%%%%%%%%%
%%%%%%%%%%%%%%%%%%%%%%%%%%%%%%%%%%%%%%%%%%%%%%%%%%%%%%%%%%%%%%%%%%%%%%%%%%%%%%%%
%%%%%%%%%%%%%%%%%%%%%%%%%%%%%%%%%%%%%%%%%%%%%%%%%%%%%%%%%%%%%%%%%%%%%%%%%%%%%%%%
%%%%%%%%%%%%%%%%%%%%%%%%%%%%%%%%%%%%%%%%%%%%%%%%%%%%%%%%%%%%%%%%%%%%%%%%%%%%%%%%
%-------------------------------------------------------------------------------
\section{Minimal rank of P$n$U non-classically integral lattices over unramified dyadic local rings}
%-------------------------------------------------------------------------------
\label{sec:peven}
%%%%%%%%%%%%%%%%%%%%%%%%%%%%%%%%%%%%%%%%%%%%%%%%%%%%%%%%%%%%%%%%%%%%%%%%%%%%%%%%
%%%%%%%%%%%%%%%%%%%%%%%%%%%%%%%%%%%%%%%%%%%%%%%%%%%%%%%%%%%%%%%%%%%%%%%%%%%%%%%%
%%%%%%%%%%%%%%%%%%%%%%%%%%%%%%%%%%%%%%%%%%%%%%%%%%%%%%%%%%%%%%%%%%%%%%%%%%%%%%%%
%%%%%%%%%%%%%%%%%%%%%%%%%%%%%%%%%%%%%%%%%%%%%%%%%%%%%%%%%%%%%%%%%%%%%%%%%%%%%%%%
%%%%%%%%%%%%%%%%%%%%%%%%%%%%%%%%%%%%%%%%%%%%%%%%%%%%%%%%%%%%%%%%%%%%%%%%%%%%%%%%
%%%%%%%%%%%%%%%%%%%%%%%%%%%%%%%%%%%%%%%%%%%%%%%%%%%%%%%%%%%%%%%%%%%%%%%%%%%%%%%%
%%%%%%%%%%%%%%%%%%%%%%%%%%%%%%%%%%%%%%%%%%%%%%%%%%%%%%%%%%%%%%%%%%%%%%%%%%%%%%%%

Hereafter, we assume that $R$ is unramified dyadic. Recall that an $R$-lattice $L$ is called non-classically integral if $\kn L\subseteq R$. A non-classically integral lattice is called primitively $n$-universal if it primitively represents all non-classically integral lattices of rank $n$. In this section, we prove that the minimal rank of primitively $n$-universal non-classically integral lattices is $2n$. Furthermore, it can be shown that, up to isometry, there exists exactly one primitively $2$-universal non-classically integral lattice of rank $4$.

Clearly, the study of primitively $n$-universal non-classically integral lattices is equivalent to the study of even lattices that primitively represent all even lattices of rank $n$. The following is a supplement of Lemma \ref{lem:easypnu} for even lattices over unramified dyadic local ring $R$.

\begin{lem}\label{lem:easypnueven}
\textup{(a)} If an $R$-lattice $J$ is isotropic, then $\bbH\mathbin{\perp} J$ primitively represents all binary lattices of the form $2^a \bbH$ for any nonnegative integer $a$.

\noindent \textup{(b)} If an $R$-lattice $J$ primitively represents $2^{a+1}\epsilon$ for a nonnegative integer $a$ and a unit $\epsilon\in R$, then $\bbH\mathbin{\perp} J$ primitively represents all binary lattices of the form $2^a \bbH$ and $2^a \bbA$. In particular, $\bbH^2$ primitively represents all binary lattices of the form $2^a \bbH$ and $2^a \bbA$ for any nonnegative integer $a$. Hence, $\bbH^n$ primitively represents all even lattices of rank $n$.
\end{lem}

\begin{proof}
We let $x, y$ be a basis for $\bbH$ such that $Q(\alpha x + \beta y) = 2\alpha\beta$ for any $\alpha$, $\beta\in R$. (a) Since $J$ is isotropic, there is a primitive vector $z\in J$ with $Q(z) = 0$. Then, we have $R[x, 2^a y + z] \cong 2^a \bbH$. (b) Let $z\in J$ be a primitive vector with $Q(z) = 2^{a+1}\epsilon$. Then, we have
\[
 R[\epsilon^{-1} x, - x + 2^a\epsilon y + z] \cong 2^a \bbH \quad\text{and}\quad
 R[x + 2^a\epsilon y, x + z] \cong 2^a\epsilon \bbA\text.
\]
Note that $2^a\epsilon \bbA \cong 2^a \bbA$ by \cite[93:11]{OM}. These facts combined with Lemma~\ref{lem:easypnu} imply the lemma.
\end{proof}

The above lemma shows that $\bbH^n$ is an $R$-lattice of rank $2n$ that primitively represents all even lattices of rank $n$. By combining it with Corollary~\ref{cor:corindgen}, we conclude that the minimal rank of even $R$-lattices that primitively represents all even $R$-lattices is exactly $2n$.

\begin{prop}
Let $R$ be an unramified dyadic local ring. Any even $R$-lattice of rank $4$ that primitively represents all even $R$-lattices of rank $2$ is isometric to $\bbH^2$.
\end{prop}

\begin{proof}
According to \cite[Theorem~1.2]{HH23} and Corollary~\ref{cor:corindgen}, any lattice on the hyperbolic space $\bbH^2$ that represents all even lattices of rank $2$ is isometric to $\bbH^2$.
\end{proof}

%%%%%%%%%%%%%%%%%%%%%%%%%%%%%%%%%%%%%%%%%%%%%%%%%%%%%%%%%%%%%%%%%%%%%%%%%%%%%%%%
%%%%%%%%%%%%%%%%%%%%%%%%%%%%%%%%%%%%%%%%%%%%%%%%%%%%%%%%%%%%%%%%%%%%%%%%%%%%%%%%
%%%%%%%%%%%%%%%%%%%%%%%%%%%%%%%%%%%%%%%%%%%%%%%%%%%%%%%%%%%%%%%%%%%%%%%%%%%%%%%%
%%%%%%%%%%%%%%%%%%%%%%%%%%%%%%%%%%%%%%%%%%%%%%%%%%%%%%%%%%%%%%%%%%%%%%%%%%%%%%%%
%%%%%%%%%%%%%%%%%%%%%%%%%%%%%%%%%%%%%%%%%%%%%%%%%%%%%%%%%%%%%%%%%%%%%%%%%%%%%%%%
%%%%%%%%%%%%%%%%%%%%%%%%%%%%%%%%%%%%%%%%%%%%%%%%%%%%%%%%%%%%%%%%%%%%%%%%%%%%%%%%
%%%%%%%%%%%%%%%%%%%%%%%%%%%%%%%%%%%%%%%%%%%%%%%%%%%%%%%%%%%%%%%%%%%%%%%%%%%%%%%%
%-------------------------------------------------------------------------------
\section{Minimal rank of P$n$U classically integral lattices over unramified dyadic local rings $R\ne\z_2$}
%-------------------------------------------------------------------------------
\label{sec:pnotz2}
%%%%%%%%%%%%%%%%%%%%%%%%%%%%%%%%%%%%%%%%%%%%%%%%%%%%%%%%%%%%%%%%%%%%%%%%%%%%%%%%
%%%%%%%%%%%%%%%%%%%%%%%%%%%%%%%%%%%%%%%%%%%%%%%%%%%%%%%%%%%%%%%%%%%%%%%%%%%%%%%%
%%%%%%%%%%%%%%%%%%%%%%%%%%%%%%%%%%%%%%%%%%%%%%%%%%%%%%%%%%%%%%%%%%%%%%%%%%%%%%%%
%%%%%%%%%%%%%%%%%%%%%%%%%%%%%%%%%%%%%%%%%%%%%%%%%%%%%%%%%%%%%%%%%%%%%%%%%%%%%%%%
%%%%%%%%%%%%%%%%%%%%%%%%%%%%%%%%%%%%%%%%%%%%%%%%%%%%%%%%%%%%%%%%%%%%%%%%%%%%%%%%
%%%%%%%%%%%%%%%%%%%%%%%%%%%%%%%%%%%%%%%%%%%%%%%%%%%%%%%%%%%%%%%%%%%%%%%%%%%%%%%%
%%%%%%%%%%%%%%%%%%%%%%%%%%%%%%%%%%%%%%%%%%%%%%%%%%%%%%%%%%%%%%%%%%%%%%%%%%%%%%%%

In this section, we continue to assume that the local ring $R$ is unramified dyadic. And we prove that
\[
 \text{If\, $R \ne \z_2$,\, then\, $u^\ast_R(2) = 5$\, and\, $u^\ast_R(n) = 2n$\, for\, $n\ge 3$.}
\]
The following are supplements of Lemma \ref{lem:easypnu} for general lattices over $R$. We mean by $Q^\ast(L)$ the set of all $Q(v)$'s, where $v$ runs over all primitive vectors in $L$.

\begin{lem}\label{lem:1-1}
Let $L\cong\ang{1, -1}$ be a binary $R$-lattice. Then, we have
\[
 Q(L) = R^\times \cup 4R\text, \qquad Q^\ast(L) = \begin{cases}
  Q(L) & \text{if $R\ne \z_2$,}\\
  R^\times \cup 8R & \text{if $R = \z_2$.}
 \end{cases}
\]
\end{lem}

\begin{proof}
Define $g(x, y) = x(x+2y)$. Since $\ang{1, -1} \cong \left(\begin{smallmatrix}1&1\\1&0\end{smallmatrix}\right)$, we have
\begin{align*}
 Q(L) & = \left\{g(x, y) \mathrel{}\middle|\mathrel{} (x, y)\in R^2\right\}\text,\\
 Q^\ast (L) & = \left\{g(x, y) \mathrel{}\middle|\mathrel{} (x, y)\in R^2\text{ is primitive}\right\}\text.
\end{align*}
If $x$ is a unit in $R$, then so is $g(x, y)$. If $x$ is even, then $g(x, y)$ is divisible by $4$. Hence, $Q(L)\subseteq R^\times\cup 4R$. Let $\epsilon$ be any unit. Then, there exists a unit $\eta$ such that $\eta^2 = \epsilon - 2\alpha$ for some $\alpha\in R$, which implies that $g(\eta, \eta^{-1}\alpha) = \epsilon$. For any nonnegative integer $a$, note that $g(2^{a+2}, \epsilon - 2^{a+1}) = 2^{a+3}\epsilon$. Suppose that $R = \z_2$. In this case, it remains to show that $Q^\ast(L)\cap 4 R^\times = \varnothing$. Assume that $g(x, y)\in 4R^\times$. Since $x$ must be even, we have $y\in R^\times$. Furthermore, since $y\equiv 1\spmod2$, this implies that $g(x, y)$ is divisible by $8$, which is absurd. Now, suppose that $R \ne \z_2$. In this case, it remains to show that $4R^\times \subseteq Q^\ast(L)$. There exists a unit $\eta$ such that $\eta^2 \not\equiv \epsilon\spmod2$. Then, $g(2\eta, \eta^{-1}\epsilon - \eta) = 4\epsilon$. This completes the proof.
\end{proof}

\begin{lem}\label{lem:easy2pnu}
\textup{(a)} For any unit $\epsilon\in R$, $\bbH\mathbin{\perp}\ang{\epsilon}$ is isometric to $\ang{1, -1, \epsilon}$. Hence, $\bbH\mathbin{\perp}\ang{\epsilon}$ primitively represents all binary $R$-lattices of the form $\ang{\alpha, \epsilon}$ for any $\alpha\in R$. In particular, $\bbH\mathbin{\perp}\ang{\epsilon}$ is primitively $1$-universal.

\noindent \textup{(b)} If an $R$-lattice $J$ primitively represents an $R$-lattice $K$ of rank $k$ as well as a unit in $R$, then $\bbH^n \mathbin{\perp} J$ primitively represents all $R$-lattices of rank $n+k$ of the form $\ell\mathbin{\perp} K$ for any $R$-lattice $\ell$ of rank $n$. In particular, $\bbH^n \mathbin{\perp} J$ is primitively $n$-universal.
\end{lem}

\begin{proof}
(a) The isometry $\bbH\mathbin{\perp}\ang{\epsilon} \cong \ang{1, -1, \epsilon}$ follows from \cite[93:16]{OM}. Now, apply Lemma~\ref{lem:1-1}. (b) If $J$ primitively represents $\epsilon\in R^\times$, then $J \cong \ang{\epsilon}\mathbin{\perp} J'$ for some $R$-lattice $J'$. Now, apply (a) inductively.
\end{proof}

By Lemma~\ref{lem:easy2pnu}, $\bbH^n\mathbin{\perp}\ang{\epsilon}$ is prmitively $n$-universal for any unit $\epsilon$ in $R$. By combining it with Corollary~\ref{cor:corindgen}, we conclude that $2n\le u^\ast_2(n)\le 2n+1$. We first settle down the case when $R\ne\z_2$. In the remaining of this section, we assume that the local ring $R$ is unramified dyadic such that $R\ne\z_2$. Note that \cite[Theorem~1.3]{HH23} implies that $u_R(2) = 5$. Thus, $u^\ast_R(2) \ge u_R(2) = 5$. Therefore, the minimal rank of primitively $2$-universal lattices is five. Next, we prove that $u^\ast_R(n) = 2n$ for any $n\ge 3$.

\begin{thm}
For any $n\ge 3$, the $R$-lattice $\bbH^{n-1} \mathbin{\perp} \ang{1, -1}$ is primitively $n$-universal. In particular, $u^\ast_R(n) = 2n$ for any $n\ge 3$.
\end{thm}

\begin{proof}
Let $L\cong \bbH^{n-1} \mathbin{\perp} \ang{1,-1}$ and let $\ell$ be any $R$-lattice of rank $n$. Since $Q^\ast(\ang{1,-1}) = R^\times \cup 4R$, Lemmas~\ref{lem:easypnueven} and \ref{lem:easy2pnu} implies that $\ell$ is primitively represented by $L$ unless $\ell$ satisfies the following condition (*):
\begin{itemize}
\item[(*)] $\ell$ is even unimodular or proper $2$-modular, or $\ell$ has a Jordan splitting $\ell\cong \ell_1 \mathbin{\perp} \ell_2$, where $\ell_1$ is even unimodular and $\ell_2$ is proper $2$-modular.
\end{itemize}
Suppose that $\ell$ satisfies the condition (*). If $\ell$ has a unimodular component, then actually $\ell \cong \bbH \mathbin{\perp} \ell'$ for some even lattice $\ell'$ of rank $n-2$. Since $\ell'$ is primitively represented by $\bbH^{n-2}$, $\ell$ is primitively represented by $L$. Otherwise, $\ell$ is proper $2$-modular. Then, $\ell$ is split by $2\bbH$ or $2\bbA$, either of which is primitively represented by $\bbH \mathbin{\perp} \ang{1,-1}$. Hence, $\ell$ is primitively represented by $L$.
\end{proof}

%%%%%%%%%%%%%%%%%%%%%%%%%%%%%%%%%%%%%%%%%%%%%%%%%%%%%%%%%%%%%%%%%%%%%%%%%%%%%%%%
%%%%%%%%%%%%%%%%%%%%%%%%%%%%%%%%%%%%%%%%%%%%%%%%%%%%%%%%%%%%%%%%%%%%%%%%%%%%%%%%
%%%%%%%%%%%%%%%%%%%%%%%%%%%%%%%%%%%%%%%%%%%%%%%%%%%%%%%%%%%%%%%%%%%%%%%%%%%%%%%%
%%%%%%%%%%%%%%%%%%%%%%%%%%%%%%%%%%%%%%%%%%%%%%%%%%%%%%%%%%%%%%%%%%%%%%%%%%%%%%%%
%%%%%%%%%%%%%%%%%%%%%%%%%%%%%%%%%%%%%%%%%%%%%%%%%%%%%%%%%%%%%%%%%%%%%%%%%%%%%%%%
%%%%%%%%%%%%%%%%%%%%%%%%%%%%%%%%%%%%%%%%%%%%%%%%%%%%%%%%%%%%%%%%%%%%%%%%%%%%%%%%
%%%%%%%%%%%%%%%%%%%%%%%%%%%%%%%%%%%%%%%%%%%%%%%%%%%%%%%%%%%%%%%%%%%%%%%%%%%%%%%%
%-------------------------------------------------------------------------------
\section{Minimal rank of P$n$U classically integral lattices over $\z_2$}
%-------------------------------------------------------------------------------
\label{sec:pz2}
%%%%%%%%%%%%%%%%%%%%%%%%%%%%%%%%%%%%%%%%%%%%%%%%%%%%%%%%%%%%%%%%%%%%%%%%%%%%%%%%
%%%%%%%%%%%%%%%%%%%%%%%%%%%%%%%%%%%%%%%%%%%%%%%%%%%%%%%%%%%%%%%%%%%%%%%%%%%%%%%%
%%%%%%%%%%%%%%%%%%%%%%%%%%%%%%%%%%%%%%%%%%%%%%%%%%%%%%%%%%%%%%%%%%%%%%%%%%%%%%%%
%%%%%%%%%%%%%%%%%%%%%%%%%%%%%%%%%%%%%%%%%%%%%%%%%%%%%%%%%%%%%%%%%%%%%%%%%%%%%%%%
%%%%%%%%%%%%%%%%%%%%%%%%%%%%%%%%%%%%%%%%%%%%%%%%%%%%%%%%%%%%%%%%%%%%%%%%%%%%%%%%
%%%%%%%%%%%%%%%%%%%%%%%%%%%%%%%%%%%%%%%%%%%%%%%%%%%%%%%%%%%%%%%%%%%%%%%%%%%%%%%%
%%%%%%%%%%%%%%%%%%%%%%%%%%%%%%%%%%%%%%%%%%%%%%%%%%%%%%%%%%%%%%%%%%%%%%%%%%%%%%%%

In this section, we prove that
\[
 u^\ast_{\z_2}(n) = \begin{cases}
  2n+1 & \text{if $n\le 4$,}\\
  2n & \text{if $n\ge 5$.}
 \end{cases}
\]
Recall that over $\z_2$, $Q(\ang{1,-1}) = \z_2^\times \cup 4\z_2$ and $Q^\ast(\ang{1,-1}) = \z_2^\times \cup 8\z_2$ by Lemma~\ref{lem:1-1}. We make a particular choice $\rho = 1$ for $R = \z_2$ so that $\mathbb{A}\cong \left(\begin{smallmatrix}2&1\\1&2\end{smallmatrix}\right)$. We know that $u_{\z_2}(2) = 5$ according to \cite[Theorem~1.3]{HH23} or \citep[Lemma~2.3]{Oh03}\footnote{It has been brought to our attention that \cite{Oh03} contains some errors. Specifically: (1) While \cite[Theorem~2.8]{Oh03} remains valid, its proof requires partial correction, and the corrected version is available. (2) Contrary to the original statement, $\langle 1, 1, 2, 4, 7 \rangle$ does not have an exceptional core $\z$-lattice. As a result, the findings in Section 3 need to be revised, and $\langle 1, 1, 2, 4, 7 \rangle$ should be recategorized as almost $2$-universal, correcting its misclassification. The full corrigendum can be provided upon request from the authors.}. Thus, we have $u^\ast_{\z_2}(2) \ge u_{\z_2}(2) = 5$. Therefore, the minimal rank of primitively $2$-universal $\z_2$-lattices is five.

Next, we prove that $u^\ast_{\z_2}(3) = 7$. Note that $6 \le u^\ast_{\z_2}(3) \le 7$. According to \cite[Theorem~6.16]{HH23} and Corollary~\ref{cor:corindgen}, any $3$-universal senary $\z_2$-lattice on the hyperbolic space $\bbH^3$ is isometric to a $\z_2$-lattice obtained from one of the following six $\z_2$-lattices by a unit scaling:

\vspace{1ex}

\begin{tasks}[label=(\Alph*),label-width=1.5em](3)
\task $\bbH^2\mathbin{\perp}\ang{1, -1}$,
\task $\bbH^2\mathbin{\perp}\ang{-1, 4}$,
\task $\bbH\mathbin{\perp}\ang{1, -1, 2, -2}$,
\task $\bbH\mathbin{\perp}\ang{1, -1, -2, 8}$,
\task $\bbH\mathbin{\perp}\ang{-1, 2, -2, 4}$,
\task $\bbH\mathbin{\perp}\ang{-1, -2, 4, 8}$.
\end{tasks}

\begin{thm}\label{thm:36}
No $\z_2$-lattice of rank $6$ is primitively $3$-universal. Therefore, we have $u^\ast_{\z_2}(3) = 7$.
\end{thm}

\begin{proof}
It suffices to show that none of the above six lattices is primitively $3$-universal. First, Let $L$ be one of (A), (B), or (E). Since $L$ is $2$-universal, $\bbA$ is represented by $L$. For any sublattice $M\cong\bbA$ of $L$, $M$ splits $L$ and
\[
 M^\perp\cong \ang{1,1,1,5}\text{, }\ang{5,1,1,4}\text{, or }\ang{5,2,2,4}
\]
by \cite[Theorem~91:9]{OM} and \cite[Corollary~93:14a]{OM}. Next, let $L$ be one of (C), (D), or (F). Then $\ang{1,3}$ is primitively represented by $L$. For any sublattice $M\cong\ang{1, 3}$ of $L$, $M$ splits $L$ and
\[
 M^\perp\cong \ang{1,3,2,-2}\text{, }\ang{1,3,-2,8}\text{, or }\ang{3,-2,4,8}
\]
by \cite[Theorem~91:9]{OM} and \cite[Corollary~93:14a]{OM}. In any case, if $L$ were primitively $3$-universal, then $M^\perp$ must be primitively $1$-universal. However, according to \cite[Theorem~5.2]{EG21jnt}, $M^\perp$ is not primitively $1$-universal. Therefore, $L$ is not primitively $3$-universal.
\end{proof}

Now, we prove that $u^\ast_{\z_2}(4) = 9$. Note that $8 \le u^\ast_{\z_2}(4) \le 9$. According to \cite[Theorem~1.3]{HH23} and Corollary~\ref{cor:corindgen}, any $4$-universal octonary $\z_2$-lattice on the hyperbolic space $\bbH^4$ is isometric to a $\z_2$-lattice obtained from one of the following $\z_2$-lattices by a unit scaling for some nonnegative integer $t$:

\vspace{1ex}

\begin{tasks}[label=(\Alph*),label-width=1.5em](2)
\task $\bbH^3\mathbin{\perp}\ang{-1, 2^{2t}}$,
\task $\bbH^2\mathbin{\perp}\ang{1, -1, -2, 2^{2t+1}}$,
\task $\bbH^2\mathbin{\perp}\ang{-1, -1, 4, 2^{2t+2}}$,
\task $\bbH^2\mathbin{\perp}\ang{-1, 2, -2, 2^{2t+2}}$,
\task $\bbH^2\mathbin{\perp}\ang{-1, -2, 4, 2^{2t+3}}$,
\task $\bbH^2\mathbin{\perp}\ang{-1, -2, 8, 2^{2t+4}}$,
\task $\bbH^2\mathbin{\perp}\ang{-1, 2, -8, 2^{2t+4}}$.
\end{tasks}

\vspace{2ex}

\begin{lem}\label{lem:no4mod8} Let $N\cong \bbA\mathbin{\perp}\ang{1, -1}$ be the $\z_2$-lattice.

\noindent \textup{(a)} No integer congruent to $4$ modulo $8$ is primitively represented by $N$.

\noindent \textup{(b)} Let $M = M_1\mathbin{\perp} \cdots\mathbin{\perp} M_r$ be a Jordan splitting of a $\z_2$-lattice $M$ such that
\[
 \z_2 \supseteq \ks M_1 \supsetneq \cdots \supsetneq \ks M_r = \kn M_r = 2^{2s}\z_2\text{,}
\]
for some positive integer $s$. Suppose that $\q_2 M \cong \q_2 N$ and $M' := M_1\mathbin{\perp} \cdots\mathbin{\perp} M_{r-1}$ is anisotropic. Then no integer congruent to $2^{2s+2}$ modulo $2^{2s+3}$ is primitively represented by $M$.
\end{lem}

\begin{proof}
The assertion (a) follows directly from Lemma~\ref{lem:1-1}. Hence, we may assume that $r\ge 2$ in (b). Let $x = x_1 + \dotsb + x_r$ be a typical vector in $M$, where $x_i\in M_i$ for each $i$ with $1\le i\le r$. If $\ord Q(x)\ge 2s$, then
\[
 \ord Q(x_1 + \dotsb + x_{r-1})\ge 2s\text.
\]
Let $N' = \{y\in M' \mid \ord Q(y)\ge 2s\}$. It suffices to show that no integer congruent to $4\cdot 4^s$ modulo $8\cdot 4^s$ is primitively represented by $N'\mathbin{\perp} M_r$. Since $N'\mathbin{\perp} M_r\cong 4^s\bbA \mathbin{\perp} \ang{4^s, -4^s}$, the assertion (b) follows directly from (a).
\end{proof}

Let $N\cong \langle \epsilon_1, \dots, \epsilon_n \rangle$ be a unimodular $\z_2$-lattice, where $\epsilon_i$ is a unit in $\z_2$ for any $i$ with $1\le i\le n$. Note that, if $N\cong \langle \delta_1, \dots, \delta_n \rangle$ for some other units $\delta_i$ in $\z_2$ for $1\le i\le n$, then
\[
 \sum_{i=1}^n \delta_i \equiv \sum_{i=1}^n \epsilon_i \mod8\text.
\]
Hence, the residue class $(\sum_{i=1}^n \epsilon_i) + 8\z_2$ contained in $\z_2/8\z_2$ is an invariant of $N$, which is called the $2$-signature of $N$. (for this, see \cite[Chapter~15]{CS}).

\begin{lem}\label{lem:eus}
Let $L$ be a $\z_2$-lattice, and let $L = L_0\mathbin{\perp} L_1\mathbin{\perp} \dotsb \mathbin{\perp} L_t$ \textup($t\ge 0$\textup) be a fixed Jordan splitting of $L$, where $L_i = 0$ or $L_i$ is $2^i$-modular for each $i$ with $0\le i\le t$. For any integer $k$ with $0\le k\le t$, define $L_{\ge k} = L_k\mathbin{\perp} L_{k+1}\mathbin{\perp} \dotsb \mathbin{\perp} L_t$. Let $x = \sum_{i=0}^t x_i$ be a vector in $L$ such that $x_i\in L_i$. Assume that $\ks L = \z_2$, $x_0$ is primitive in $L_0$, and $Q(x)$ is even.

\noindent \textup{(a)} Suppose that $\kn L = \z_2$. Write $L_0 \cong \langle \epsilon_1, \dots, \epsilon_n \rangle$, where $\epsilon_i$ are units in $\z_2$ for $1\le i\le n$. Let $\gamma = \sum_{i=1}^n \epsilon_i$ be the $2$-signature of $L_0$. Assume that one of the following conditions holds:
\begin{enumerate}[\textup{(\roman{enumi})}]
\item $\kn (L_{\ge 1}) \subseteq 8\z_2$ and $Q(x) \not\equiv \gamma \spmod8$;
\item $\kn (L_{\ge 1}) \subseteq 4\z_2$ and $Q(x) \not\equiv \gamma \spmod4$;
\item $L_1 \cong \ang{2\epsilon}$, $\kn (L_{\ge 2}) \subseteq 8\z_2$, and $Q(x) - \gamma \not\equiv 0$, $2\epsilon \spmod8$;
\item $L_1 \cong \ang{2\epsilon, 2\epsilon'}$, $dL_1 \equiv 4\spmod{16}$, $\kn (L_{\ge 2}) \subseteq 8\z_2$, and
\[
 Q(x) - \gamma \not\equiv 0\text{, }2\epsilon\text{, }4 \pmod8\text;
\]
\item $L_1 \cong \ang{2\epsilon, 2\epsilon'}$, $dL_1 \equiv -4\spmod{16}$, $\kn (L_{\ge 2}) \subseteq 8\z_2$, and
\[
 Q(x) - \gamma \not\equiv 0\text{, }{\pm 2} \pmod8\text;
\]
\item $Q(x) \not\equiv \gamma \spmod2$.
\end{enumerate}
Then, there exists a binary even unimodular sublattice $M$ of $L$ containing $x$. In particular, $\rank L_0\ge 3$.

\noindent \textup{(b)} In addition to the assumptions in \textup{(a)}, suppose further that $Q(x)\equiv 2\spmod4$. Then, in addition to the conclusions in \textup{(a)}, there exists another binary even unimodular sublattice $M'$ of $L$ containing $x$ such that $M'\not\cong M$ if and only if $\rank L_0\ge 5$, $\kn (L_{\ge 1}) = 2\z_2$, or $\rank L_0 = 4$ and $dL_0 \equiv 3\spmod4$.

\noindent \textup{(c)} Suppose that $\kn L \subseteq 2\z_2$. Then, there exists a binary even unimodular sublattice $M$ of $L$ containing $x$. Suppose further that $Q(x)\equiv 2\spmod4$. Then, there exists another binary even unimodular sublattice $M'$ of $L$ containing $x$ such that $M'\not\cong M$ if and only if $\rank L_0\ge 4$ or $\kn (L_{\ge 1}) = 2\z_2$.
\end{lem}

\begin{proof}
(a) Suppose that $L_0 \cong \ang{\epsilon_1, \dotsc, \epsilon_n}$ in $e_1, \dotsc, e_n$. Write $x_0 = \xi_1 e_1 + \dotsb + \xi_n e_n$. Since $x_0$ is primitive, we may assume that $\xi_1$ is odd. We claim that at least one among $\xi_2$, \dots, $\xi_n$ is even. Suppose on the contrary that $\xi_2 \dotsm \xi_n$ is odd. Note that
\[
 Q(x_0) = \epsilon_1 \xi_1^2 + \epsilon_2 \xi_2^2 + \dotsb + \epsilon_n \xi_n^2 \equiv s(L_0) \pmod8\text,
\]
which contradicts the hypothesis (i) through (vi). Hence, we may assume that $\xi_n$ is even. Now, $M = \z_2[x, e_1 + e_n]$ is an even unimodular binary $\z_2$-lattice containing $x$.

(b) We continue to suppose that $L_0 \cong \ang{\epsilon_1, \dotsc, \epsilon_n}$ in $e_1, \dotsc, e_n$, and $x_0 = \xi_1 e_1 + \dotsb + \xi_n e_n$. By renumbering indices if necessary, we may assume that $\xi_1 \dotsm \xi_s$ is odd and $\xi_{s+1} \equiv \cdots \equiv \xi_n \equiv 0\spmod2$ for some $s$ with $1\le s < n$. As in (a), we let $M = \z_2[x, e_1 + e_n]$. First, we assume that $\kn (L_{\ge 1}) = 2\z_2$. Then, there exists $y\in L_{\ge 1}$ with $Q(y) \equiv 2\spmod4$. Note that $B(x, y) = B(x-x_0, y) \equiv 0\spmod2$. Hence, $M$ and $M' = \z_2[x, e_1 + e_n + y]$ satisfy all the required conditions. From now on, we assume that $\kn (L_{\ge 1}) \subseteq 4\z_2$.

Assume that $\rank L_0 \ge 5$ or $\rank L_0 = 4$ and $dL_0 \equiv 3\spmod4$. If there exists an index $i$ with $1 < i \le s$ such that $\epsilon_1 \not\equiv \epsilon_i\spmod4$, then $M$ and $M' = \z_2[x, e_i + e_n]$ satisfy all the desired properties. The same conclusion can be obtained for the case when there exists an index $j$ with $s+1\le j < n$ such that $\epsilon_j \not\equiv \epsilon_n\spmod4$. Now, we assume that
\[
 \epsilon_1 \equiv \cdots \equiv \epsilon_s \spmod4 \quad \text{and} \quad \epsilon_{s+1} \equiv \cdots \equiv \epsilon_n \spmod4\text.
\]
If $\rank L_0 = 4$, then the above conditions on $\epsilon_i$'s imply that $dL_0\equiv 1\spmod4$, which is absurd. Hence, we must have $\rank L_0\ge 5$, which implies that either $s\ge 3$ or $n-s\ge 3$. If we let
\[
 M' = \begin{cases}
  \z_2[x, e_1 + e_2 + e_3 + e_n] & \text{if $s\ge 3$},\\
  \z_2[x, e_1 + e_{n-2} + e_{n-1} + e_n] & \text{if $n-s\ge 3$},
 \end{cases}
\]
then $M$ and $M'$ satisfy all the required conditions.

Finally, assume that $\rank L_0 = 4$ and $dL_0 \equiv 1\spmod4$ or $\rank L_0 = 3$. Suppose that there exists an even unimodular binary sublattice $M'$ of $L$. Since we are assuming that $\kn (L_{\ge 1}) \subseteq 4\z_2$, $M'$ is represented by $L_0$ by virtue of \cite[Proposition~20]{OM58}. Note that
\[
 \bbH \mathbin{\perp} \ang{1, -1} \not\cong \bbA \mathbin{\perp} \ang{1, 3}\text, \quad \bbH \mathbin{\perp} \ang{1, 3} \not\cong \bbA \mathbin{\perp} \ang{1, -1}\text,
\]
and $\bbH \mathbin{\perp} \ang{\epsilon} \not\cong \bbA \mathbin{\perp} \ang{5\epsilon}$ for any $\epsilon \in \z_2^\times$. Hence, if $M'$ is represented by $L_0$, then $M'\cong M$. This proves the lemma.

(c) Suppose that $L_0 \cong J_1 \mathbin{\perp} \dotsb \mathbin{\perp} J_n$, where $J_i \cong \bbH$ or $\bbA$ in $e_{2i-1}, e_{2i}$ for $1\le i\le n$. Write $x_0 = \xi_1 e_1 + \dotsb + \xi_{2n} e_{2n}$, and assume that $\xi_1$ is odd. Let $M = \z_2[x, e_2]$. Then, $M$ satisfies all the desired properties for the former assertion. Now, suppose that $Q(x)\equiv 2\spmod4$. First, assume that $\kn (L_{\ge 1}) = 2\z_2$. Then, there exists $y\in L_{\ge 1}$ with $Q(y) \equiv 2\spmod4$. Note that $B(x, y) = B(x-x_0, y) \equiv 0\spmod2$. Hence, $M$ and $M' = \z_2[x, e_2 + y]$ satisfy all the required conditions. From now on, we assume that $\kn (L_{\ge 1}) \subseteq 4\z_2$.

Assume that $\rank L_0\ge 4$. If $\xi_3 \equiv \xi_4 \equiv 0\spmod2$, then $M$ and $M' = \z_2[x, e_2 + e_3 + e_4]$ satisfy all the desired properties. Hence, we may assume that $\xi_3$ is odd. If $J_1 \not\cong J_2$, then $M$ and $M' = \z_2[x, e_4]$ satisfy all the required conditions. If $J_1 \cong J_2$, then the same conclusion can be obtained by considering the isometry $\bbH \mathbin{\perp} \bbH \cong \bbA \mathbin{\perp} \bbA$. Now, assume that $\rank L_0 = 2$. Suppose that there exists an even binary unimodular sublattice $M'$ of $L$. Since we are assuming that $\kn (L_{\ge 1}) \subseteq 4\z_2$, $M'$ is represented by $L_0$ by \cite[Proposition~20]{OM58}. Hence, $M' \cong L_0 \cong M$. This proves the lemma.
\end{proof}

\begin{lem}\label{lem:isoiso}
Let $L$ be a binary $\z_2$-lattice such that $\ks L = \z_2$ and $\q_2 L \cong \bbH$. Let $s$ and $t$ be integers such that $t > 2s = \ord_2 dL$. Let $z$ be a primitive vector in $L$ such that $\ord_2 Q(z) = t+1$, and let $w$ be a vector in $L$ such that $\ord_2 B(z, w) = t$. Then, $\ord_2 Q(w) \ge t+2$.
\end{lem}

\begin{proof}
First, we claim that $\ord_2 Q(w) \ge t$. Assume that $L$ is even. Let $e_1, e_2$ be a basis for $L$ such that $L \cong \bbH$ in $e_1, e_2$. Write $z = z_1 e_1 + z_2 e_2$. We may assume that $z_1 = 1$. Then, $z_2 \equiv 2^t \spmod{2^{t+1}}$. Write $w = w_1 e_1 + w_2 e_2$. Since
\[
 0 \equiv B(z, w) = z_1 w_2 + z_2 w_1 \equiv w_2 \pmod{2^t}\text,
\]
we have $Q(w) \equiv 0 \spmod{2^{t+1}}$, as desired. Now, we assume that $\kn L = \z_2$. Let $e_1, e_2$ be a basis for $L$ such that $L \cong \left(\begin{smallmatrix}\epsilon & \alpha \\ \alpha & 0\end{smallmatrix}\right)$ in $e_1, e_2$, where $\epsilon$ is a unit in $\z_2$ and $\alpha$ is an integer in $\z_2$ such that $\ord_2 \alpha = s$. Note that $Q(\zeta_1 e_1 + \zeta_2 e_2) = \zeta_1(\epsilon\zeta_1 + 2\alpha\zeta_2)$. Hence,
\[
 \ord_2 Q(\zeta_1 e_1 + \zeta_2 e_2) \begin{cases}
  = 2\ord_2 \zeta_1 & \text{if $\ord_2 \zeta_1 \le s$,}\\
  \ge 2s+2 & \text{if $\ord_2 \zeta_1 \ge s+1$.}
 \end{cases}
\]
Write $z = z_1 e_1 + z_2 e_2$. Since $Q(z)$ is even, $z_1$ cannot be a unit. This implies that $z_2$ is a unit, and we have
\[
 \ord_2 Q(z) \begin{cases}
  = 2\ord_2 z_1 & \text{if $\ord_2 z_1 \le s$,}\\
  \ge 2s+3 & \text{if $\ord_2 z_1 = s+1$,}\\
  = s+1 + \ord_2 z_1 & \text{if $\ord_2 z_1 \ge s+2$.}
 \end{cases}
\]
Since $t > 2s$, we have $\ord_2 z_1 \ge s+1$ and $t \ge 2s+2$. Note that we have $\ord_2 z_1 = s+1$ or $t-s$. If $\ord_2 z_1 = s+1$, then $\ord_2 (\epsilon z_1 + 2\alpha z_2) = t-s$. Hence, by changing basis from $e_1, e_2$ to $e_1, -2\alpha e_1 + \epsilon e_2$ and converting coordinates $z_1$, $z_2$, $w_1$, $w_2$ accordingly if necessary, we may assume that $\ord_2 z_1 = t-s$. Write $w = w_1 e_1 + w_2 e_2$. Note that $B(z, w) = \epsilon z_1 w_1 + \alpha (z_1 w_2 + z_2 w_1)$ and $\ord_2(\epsilon z_1 + \alpha z_2) = s$. Since
\[
 0 \equiv B(z, w) \equiv (\epsilon z_1 + \alpha z_2) w_1 \pmod{2^t}\text,
\]
we have $w_1 \equiv 0 \spmod{2^{t-s}}$. Hence, $\ord_2 Q(w) = \ord_2 w_1 + \ord_2(\epsilon w_1 + 2\alpha w_2) \ge t$, as desired.

Now, assuming that $\ord_2 Q(w) \ge t$, we prove that $\ord_2 Q(w) \ge t+2$. Write $Q(z) = 2^{t+1} \epsilon$ and $Q(w) = 2^t \alpha$ for some $\epsilon \in \z_2^\times$ and $\alpha \in \z_2$. Then, $d(\z_2[z, w]) \equiv 2^{2t}(2\epsilon\alpha - 1) \spmod{2^{2t+3}}$. Since $z$ and $w$ are linearly independent, $\q_2[z, w] \cong \bbH$, which implies that $d(\q_2[z, w]) = -1$. Hence, we must have $\alpha \equiv 0\spmod4$, and the result follows directly from this.
\end{proof}

\begin{thm}\label{thm:48}
No $\z_2$-lattice of rank $8$ is primitively $4$-universal. Therefore, we have $u^\ast_{\z_2}(4) = 9$.
\end{thm}

\begin{proof}
It suffices to show that the $\z_2$-lattice $\ell$ is not primitively represented by the $\z_2$-lattice $L$ for each pair $(\ell, L)$ which is given in the table below. Let $(u, M)$ be the pair given in the same row with the pair $(\ell, L)$ in the table. Let $\ell = \ell_0 \mathbin{\perp} \ell'$, where $\ell_0$ is the unimodular component of $\ell$, and $\ell' \cong 2^{2t+u+1}\bbA$. For any representation $\phi : \ell_0 \to L$, the orthogonal complement of $\phi(\ell_0)$ in $L$ is always isometric to $M$ by cancellation laws (see \cite[Proposition~93:14a]{OM} and \cite[Theorem~5.3.6]{Ki}), independently of the choice of the representation $\phi$. Hence, it suffices to show that the $\z_2$-lattice $\ell'$ is not primitively represented by the $\z_2$-lattice $M$.

\begin{center}
\noindent\begin{tabular}{|c|c|c|c|c|}
 \hline
  & $\ell$ & $L$ & $M$ & $u$\\
 \hline
 (A) & $\bbA\mathbin{\perp} 2^{2t+1}\bbA$ & $\bbH^3\mathbin{\perp}\ang{-1, 2^{2t}}$ & $\bbH\mathbin{\perp}\ang{1, 1, 5, 2^{2t}}$ & $0$\\
 \hline
 (B) & $\ang{1,3}\mathbin{\perp} 2^{2t+2}\bbA$ & $\bbH^2\mathbin{\perp}\ang{1,-1,-2, 2^{2t+1}}$ & $\bbH\mathbin{\perp}\ang{1, 1, 10, 2^{2t+1}}$ & $1$\\
 \hline
 (C) & $\bbA\mathbin{\perp} 2^{2t+3}\bbA$ & $\bbH^2\mathbin{\perp}\ang{-1,-1,4, 2^{2t+2}}$ & $\bbH\mathbin{\perp}\ang{1, 1, 20, 2^{2t+2}}$ & $2$\\
 \hline
 (D) & $\bbA\mathbin{\perp} 2^{2t+3}\bbA$ & $\bbH^2\mathbin{\perp}\ang{-1, 2, -2, 2^{2t+2}}$ & $\bbH\mathbin{\perp}\ang{5, 2, 2, 2^{2t+2}}$ & $2$\\
 \hline
 (E) & $\ang{1,3}\mathbin{\perp} 2^{2t+4}\bbA$ & $\bbH^2\mathbin{\perp}\ang{-1, -2, 4, 2^{2t+3}}$ & $\bbH\mathbin{\perp}\ang{1, 10, 4, 2^{2t+3}}$ & $3$\\
 \hline
 (F) & $\bbA\mathbin{\perp} 2^{2t+5}\bbA$ & $\bbH^2\mathbin{\perp}\ang{-1, -2, 8, 2^{2t+4}}$ & $\bbH\mathbin{\perp}\ang{5, 2, 8, 2^{2t+4}}$ & $4$\\
 \hline
 (G) & $\bbA\mathbin{\perp} 2^{2t+5}\bbA$ & $\bbH^2\mathbin{\perp}\ang{-1, 2, -8, 2^{2t+4}}$ & $\bbH\mathbin{\perp}\ang{1, 6, -8, 2^{2t+4}}$ & $4$\\
 \hline
\end{tabular}
\end{center}

Suppose that there exists a primitive representation from $\ell'$ to $M$. Then, there are vectors $z$, $w\in M$ such that
\[
 Q(z) = Q(w) = 2^{2t+u+2}\text,\quad B(z, w) = 2^{2t+u+1}\text,
\]
and $\z_2[z, w]$ is a primitive sublattice of $M$.

Suppoose that there exists a basis $e_1, \dots, e_6$ for $M$ that satisfies all of the following conditions:

\begin{enumerate}[(\roman{enumi})]
\item $\ord_2 Q(e_6) = 2t+u$, and $\z_2 e_6$ splits $M$;
\item $z = z_1 e_1 + z_2 e_2$, where $J = \z_2[e_1, e_2]$ is isotropic, splits $M$, and satisfies
\[
 2t+u+1 > \ord_2 dJ - \ord_2 \ks J\text;
\]
\item For $w = \sum_{i=1}^6 w_i e_i$, the vector $\sum_{i=3}^6 w_i e_i$ is primitive in $M$.
\end{enumerate}
By the condition (ii) and Lemma~\ref{lem:isoiso}, we have
\[
 Q(w_1 e_1 + w_2 e_2) \equiv 0\spmod{2^{2t+u+3}}\text.
\]
Hence, by the condition (iii), $\sum_{i=3}^6 w_i e_i$ is a primitive vector in $\z_2[e_3, \dots, e_6]$ such that $Q(\sum_{i=3}^6 w_i e_i) \equiv 2^{2t+u+2}\spmod{2^{2t+u+3}}$. However, this is impossible by the condition (i) and Lemma \ref{lem:no4mod8}. This proves the theorem.

Now, we prove the existence of a basis for $M$ that satisfies the conditions (i), (ii), and (iii). Since all the other cases can be proved in similar manners, we only provide the proof of the case (C). Assume that $M\cong\ang{1, 1, 1, 3, 4, 2^{2t+2}}$ in $e_1, \dots, e_6$. Write $z = \sum_{i=1}^6 z'_i e'_i$ and $w = \sum_{i=1}^6 w'_i e'_i$. First, assume that $t=0$. Write $z^\dag = z'_5 e'_5 + z'_6 e'_6$, $w^\dag = w'_5 e'_5 + w'_6 e'_6$. Suppose if possible that $\z_2[z^\dag, w^\dag]$ is primitive. Since $\z_2[z^\dag, w^\dag] \cong \ang{4, 4}$, exactly one among the following conditions holds:
\begin{enumerate}[(\roman{enumi})]
 \item $Q(z^\dag)\equiv 4\spmod{32}$, $B(z^\dag, w^\dag)\equiv 0\spmod{16}$, $Q(w^\dag)\equiv 4\spmod{32}$;
 \item $Q(z^\dag)\equiv 4\spmod{32}$, $B(z^\dag, w^\dag)\equiv 8\spmod{16}$, $Q(w^\dag)\equiv 20\spmod{32}$;
 \item $Q(z^\dag)\equiv 20\spmod{32}$, $B(z^\dag, w^\dag)\equiv 8\spmod{16}$, $Q(w^\dag)\equiv 4\spmod{32}$;
 \item $Q(z^\dag)\equiv 20\spmod{32}$, $B(z^\dag, w^\dag)\equiv 0\spmod{16}$, $Q(w^\dag)\equiv 20\spmod{32}$;
 \item $Q(z^\dag)\equiv 4\spmod{16}$, $B(z^\dag, w^\dag)\equiv 4\spmod8$, $Q(w^\dag)\equiv 8\spmod{32}$;
 \item $Q(z^\dag)\equiv 8\spmod{32}$, $B(z^\dag, w^\dag)\equiv 4\spmod8$, $Q(w^\dag)\equiv 4\spmod{16}$.
\end{enumerate}
Since
\begin{multline*}
 \quad\quad\quad\begin{pmatrix} Q(z-z^\dag) & B(z-z^\dag, w-w^\dag) \\ B(z-z^\dag, w-z^\dag) & Q(w-w^\dag) \end{pmatrix}\\
 = \begin{pmatrix} Q(z) & B(z, w) \\ B(z, w) & Q(w) \end{pmatrix} - \begin{pmatrix} Q(z^\dag) & B(z^\dag, w^\dag) \\ B(z^\dag, w^\dag) & Q(w^\dag) \end{pmatrix}\text,\quad\quad\quad
\end{multline*}
in any case, we must have $\z_2[z-z^\dag, w-w^\dag] \cong \ang{12, -4}$, which contradicts the fact that $\ang{3, -1}$ is not represented by $\ang{1, 1, 1, 3}$ over $\q_2$. Hence, we may assume that at least one among $z'_1$, \dots, $z'_4$ is odd. Note that $Q(z) \equiv 0 \not\equiv 1+1+1+3 \spmod4$. Hence, we may apply Lemma~\ref{lem:eus} to $M$ and $z$ to obtain a basis $e_1, \dots, e_6$ for $M$ such that $M \cong \bbH \mathbin{\perp} \ang{1, 5, 4, 4}$ in the basis and such that $z = z_1 e_1 + z_2 e_2$. We may assume that $z_1 = 1$. Then, $z_2$ is even. Write $w = \sum_{i=1}^6 w_i e_i$. Since
\[
 B(z, w) = z_1 w_2 + z_2 w_1
\]
is even, $w_2$ also is even. Therefore, $e_1, \dots, e_6$ is the basis which we are seeking for. Next, assume that $t\ge 1$. We may assume that at least one among $z'_1$, \dots, $z'_5$ is odd. If $z'_i$ is odd for some $1\le i\le 4$, then we may apply Lemma~\ref{lem:eus} to $M$ and $z$ to obtain the basis which we are seeking for. Suppose that $z'_1 \equiv \cdots \equiv z'_4 \equiv 0\spmod2$ and $z'_5$ is odd. Let $e_1 = z'_5 e'_5 + z'_6 e'_6$ and $e_2 = \frac12(z - e_1)$. Since $z'_5$ is odd, we have $Q(e_1) \equiv 4\spmod{16}$. Hence, $\z_2 e_1$ splits $N^\dag = \z_2[e'_5, e'_6]$. Note that
\[
 Q(e_1)\ \equiv\ 4\ \text{ or }\ 20\pmod{32}\text{,}
\]
where the latter may occur only when $t=1$. Thus, there exists a basis $e_1, e_6$ for $N^\dag$ such that $N^\dag \cong \ang{Q(e_1), 2^{2t+2}}$ or $\ang{Q(e_1), 80}$ in $e_1, e_6$. Since $Q(e_1) \equiv 4\spmod{16}$, we have $Q(e_2) \equiv -1\spmod4$. Since $Q(e_1) + 4Q(e_2) = Q(e_1 + 2e_2) = Q(z)\equiv 0\spmod{32}$, we have
\[
 Q(e_2)\ \equiv\ -1\ \text{ or }\ 3\pmod8\text{,}
\]
according as $Q(e_1) \equiv 4$ or $20\spmod{32}$. Hence, in any case, $\z_2 e_2$ splits $N = \z_2[e'_1, \dots, e'_4]$. Thus, there exists a basis $e_2, \dots, e_5$ for $N$ such that $N \cong \ang{Q(e_2), 1, 1, 5}$ or $\ang{Q(e_2), 1, 1, 1}$ in this basis. So far, we found a basis $e_1, \dots, e_6$ for $M$ such that
\[
 M\ \cong\ \ang{Q(e_1), Q(e_2), 1, 1, 5, 2^{2t+2}}\ \text{ or }\ \ang{Q(e_1), Q(e_2), 1, 1, 1, 80}
\]
in this basis and such that $z = e_1 + 2 e_2$. Note that $2t+u+1 \ge 5 > 2 = \ord_2 Q(e_1)$. Write $w = \sum_{i=1}^6 w_i e_i$. Since $B(z, w) = Q(e_1) w_1 + 2 Q(e_2) w_2 \equiv 0\spmod4$, $w_2$ also is even. Therefore, $e_1, \dots, e_6$ is the basis which we are seeking for.
\end{proof}

Finally, we prove that $u^\ast_{\z_2}(n) = 2n$ for any $n\ge 5$. Recall that, for a local ring $R$ and for an $R$-lattice $L$, $Q^\ast(L)$ means the set of all $Q(v)$, where $v$ runs over all primitive vectors in $L$.

\begin{lem}\label{lem:easypnucap}
Let $R$ be an unramified dyadic local ring and let $\ell$ be an even $R$-lattice of rank $n\ge 1$. If $J$ is an $R$-lattice such that $Q^\ast(\ell) \cap Q^\ast(J) \ne \varnothing$, then $\ell$ is primitively represented by $\bbH^{n-1} \mathbin{\perp} J$.
\end{lem}

\begin{proof}
Assume that $\ell \cong K_1 \mathbin{\perp} \dotsb \mathbin{\perp} K_r$, where $K_i$ is either a (proper modular) lattice of rank $1$ or an improper modular lattice of rank $2$ for each $i$ with $1\le i\le r$. Let $s_0 = 0$ and let $s_i = (\rank K_1) + \dotsb + (\rank K_i)$ for $1\le i\le r$. Note that $s_r = n$. Fix $\beta\in Q^\ast(\ell)\cap Q^\ast(J)$ and find primitive vectors $w\in \ell$ and $z\in J$ such that $Q(w) = \beta = Q(z)$. Write $w = w_1 + \dotsb + w_r$, where $w_i \in K_i$ for $1\le i\le r$. By renumbering indices if necessary, we may assume that $w_r$ is primitive in $K_r$.

Assume that $L \cong \bbH^{n-1} \mathbin{\perp} J$ in $e_1, e_2, \dots, e_m$. Note that $K_i$ is isometric to one of $\langle 2\alpha \rangle$, $\alpha\bbH$, or $\alpha\bbA$ for some integer $\alpha\in R$ for $1\le i\le r$. We define vectors $x_j$, $y_j\in L$ and integers $\beta_j \in R$ for $1\le j\le n-1$ as follows. Let $1\le i\le r-1$. First, suppose that $K_i \cong \langle 2\alpha \rangle$. Then, let
\begin{align*}
 x_{s_i} & = e_{2s_i - 1} + \alpha e_{2s_i}\text, & y_{s_i} & = e_{2s_i - 1} - \alpha e_{2s_i}\text.
\end{align*}
Note that $B(x_{s_i}, y_{s_i}) = 0$ and there is $\sigma: Rx_{s_i} \cong K_i$. Write $w_i = \beta_{s_i} \sigma(x_{s_i})$. Note that $Q(\beta_{s_i} y_{s_i}) = -Q(w_i)$. Now, suppose that $K_i \cong \alpha\bbH$. Then, let
\begin{align*}
 x_{s_i-1} & = e_{2s_i - 3}\text, & y_{s_i-1} & = e_{2s_i - 3} - \alpha e_{2s_i}\text,\\
 x_{s_i} & = \alpha e_{2s_i - 2} + e_{2s_i - 1}\text, & y_{s_i} & = e_{2s_i - 1}\text.
\end{align*}
Note that $B(R[x_{s_i-1}, x_{s_i}], R[y_{s_i-1}, y_{s_i}]) = 0$ and there is $\sigma: R[x_{s_i-1}, x_{s_i}] \cong K_i$. Write $w_i = \beta_{s_i-1} \sigma(x_{s_i-1}) + \beta_{s_i} \sigma(x_{s_i})$. Note that $Q(\beta_{s_i-1} y_{s_i-1} + \beta_{s_i} y_{s_i}) = -Q(w_i)$. Next, suppose that $K_i \cong \alpha\bbA$. Then, let
\begin{align*}
 x_{s_i-1} & = e_{2s_i - 3} + \alpha e_{2s_i - 2}\text, & y_{s_i-1} & = e_{2s_i - 3} - \alpha e_{2s_i - 2} + \alpha e_{2s_i}\text,\\
 x_{s_i} & = e_{2s_i - 3} + e_{2s_i - 1} + \rho\alpha e_{2s_i}\text, & y_{s_i} & = - e_{2s_i - 1} + \rho\alpha e_{2s_i}\text.
\end{align*}
Note that $B(R[x_{s_i-1}, x_{s_i}], R[y_{s_i-1}, y_{s_i}]) = 0$ and there is $\sigma: R[x_{s_i-1}, x_{s_i}] \cong K_i$. Write $w_i = \beta_{s_i-1} \sigma(x_{s_i-1}) + \beta_{s_i} \sigma(x_{s_i})$. Note that $Q(\beta_{s_i-1} y_{s_i-1} + \beta_{s_i} y_{s_i}) = -Q(w_i)$. Finally, if $K_r \cong \alpha\bbH$, then let $x_{n-1} = e_{2n-3}$, $y_{n-1} = \alpha e_{2n-2}$, and $\beta_{n-1} = 1$. If $K_r \cong \alpha\bbA$, note that $Q(w_r) = 2\epsilon\alpha$ for some unit $\epsilon$ in $R$. In that case, let $x_{n-1} = \rho \epsilon^{-1} e_{2n-3} + \alpha e_{2n-2}$, $y_{n-1} = e_{2n-3}$, and $\beta_{n-1} = 1$. Now, note that
\[
 R[x_1, \dots, x_{n-1}, \beta_1 y_1 + \dotsb + \beta_{n-1} y_{n-1} + z]
\]
is a primitive sublattice of $L$ that is isometric to $\ell$.
\end{proof}

\begin{lem}\label{lem:48excep}
The $\z_2$-lattice $\bbH^3\mathbin{\perp}\ang{1, -1}$ primitively represents all $\z_2$-lattices of rank $4$ except for the quaternary $\z_2$-lattice $\bbA\mathbin{\perp} 2\bbA$.
\end{lem}

\begin{proof}
Let $L \cong \bbH^3\mathbin{\perp}\ang{1, -1}$ and let $\ell$ be any $\z_2$-lattice of rank $4$ not isometric to $\bbA\mathbin{\perp} 2\bbA$. Since $Q^\ast(\ang{1, -1}) = \z_2^\times \cup 8\z_2$, $\ell$ is primitively represented by $L$ unless any Jordan decomposition $\ell = \ell_1 \mathbin{\perp} \dotsb \mathbin{\perp} \ell_t$ ($\ell_i$ is nonzero for any $1\le i\le t$) of $\ell$ satisfies $2\z_2 \supseteq \kn\ell_1 \supseteq \cdots \supseteq \kn\ell_t \supseteq 4\z_2$, by Lemmas~\ref{lem:easypnueven} and \ref{lem:easy2pnu}. Suppose that $\ell$ has such a Jordan decomposition. Assume that $\ell$ is orthogonally split by $2^a\bbH$ for some nonnegative integer $a$. Since $2^a\bbH$ is primitively represented by $\bbH\mathbin{\perp}\ang{1, -1}$, $\ell$ is primitively represented by $L$. Hence, we may assume that $\ell$ has a proper component. Assume that $\ell \cong \ang{2^{a_1}\epsilon_1, 2^{a_2}\epsilon_2, 2^{a_3}\epsilon_3, 2^{a_4}\epsilon_4}$. Since $a_1, a_2, a_3, a_4\in\{1, 2\}$, up to rearrangement, at least one among
\[
 2^{a_1}\epsilon_1 + 2^{a_2}\epsilon_2\text{,} \quad 2^{a_1}\epsilon_1 + 2^{a_2}\epsilon_2 + 2^{a_3}\epsilon_3\text{,} \quad 2^{a_1}\epsilon_1 + 2^{a_2}\epsilon_2 + 2^{a_3}\epsilon_3 + 2^{a_4}\epsilon_4
\]
is a multiple of $8$, so that it is primitively represented by $\ang{1, -1}$. Hence, $\ell$ is primitively represented by $L$ by Lemma~\ref{lem:easypnucap}. Finally, assume that $\ell \cong \ang{2^{a_1}\epsilon_1, 2^{a_2}\epsilon_2}\mathbin{\perp} 2^{a_3}\bbA$. Note that $a_1, a_2\in\{1, 2\}$ and $a_3\in\{0, 1\}$. If $a_1 = a_3 + 1$, then $\ell$ is split by $2^{a_3}\bbH$. If $a_1 = a_3 = 1$, then $\ell$ is an orthogonal sum of proper components. Hence, we may assume that $a_1 = a_2 = 2$ and $a_3 = 0$. Then, $2^{a_1}\epsilon_1 + 2^{a_2}\epsilon_2$ is a multiple of $8$. Hence, $\ell$ is primitively represented by $L$ by Lemma~\ref{lem:easypnucap}.
\end{proof}

\begin{thm}
For any $n\ge 5$, the $\z_2$-lattice $\bbH^{n-1} \mathbin{\perp} \ang{1, -1}$ is primitively $n$-universal. In particular, $u^\ast_{\z_2}(n) = 2n$ for any $n\ge 5$.
\end{thm}

\begin{proof}
Suppose that $n=5$. Any $\z_2$-lattice of rank $5$ is split by a $\z_2$-lattice of rank $1$. Hence, by Lemma \ref{lem:48excep}, it suffices to show that any $\z_2$-lattice $\ell$ of the form $\bbA\mathbin{\perp} 2\bbA\mathbin{\perp}\ang{2^a\epsilon}$ is primitively represented by $\bbH^4\mathbin{\perp}\ang{1,-1}$. If $a=0$ or $a\ge 3$, then $\ang{2^a\epsilon}$ is primitively represented by $\ang{1,-1}$. If $a = 2$, then $\ell\cong \bbA\mathbin{\perp} 2\bbH\mathbin{\perp}\ang{20\epsilon}$. If $a = 1$, then $\ell\cong \bbH\mathbin{\perp} 2\bbA\mathbin{\perp}\ang{10\epsilon}$.

Suppose that $n=6$. Any $\z_2$-lattice of rank $6$ either has a component of rank $1$ or is an orthogonal sum of improper $\z_2$-lattices of rank $2$. For the former case, the primitive representability follows directly from the case when $n=5$ by Lemma~\ref{lem:easy2pnu}. For the latter case, consider a $\z_2$-lattice of the form $\ell_1\mathbin{\perp} \ell_2\mathbin{\perp} \ell_3$, where $\ell_i$'s are binary improper $\z_2$-lattices. Clearly, it is impossible that $\ell_1\mathbin{\perp} \ell_2 \cong \ell_1\mathbin{\perp} \ell_3 \cong \ell_2\mathbin{\perp} \ell_3 \cong \bbA\mathbin{\perp} 2\bbA$. Hence, we may assume that $\ell_1\mathbin{\perp} \ell_2 \not\cong \bbA\mathbin{\perp} 2\bbA$. Therefore, $\ell_1\mathbin{\perp} \ell_2$ is primitively represented by $\bbH^3\mathbin{\perp}\ang{1,-1}$, and $\ell_3$ is primitively represented by $\bbH^2$.

The case when $n\ge 7$ follows by an induction on $n$. It follows from the case of $n-1$ for any odd integer $n$, from the cases of $n-1$ and $n-2$ for any even integer $n$, by Lemmas~\ref{lem:easypnueven} and \ref{lem:easy2pnu}. This completes the proof.
\end{proof}

\end{document}